\documentclass[a4paper]{article}

\usepackage{hyperref}
\hypersetup{
  colorlinks   = true, 
  urlcolor     = blue, 
  linkcolor    = blue, 
  citecolor   = red 
}
\usepackage{amsmath}
\usepackage{aliascnt,cleveref}
\usepackage{amsthm}
\usepackage{amssymb}
\usepackage{enumitem}

\theoremstyle{plain}
\newtheorem{theorem}{Theorem}[section]
\newaliascnt{lemma}{theorem}
\newtheorem{lemma}[lemma]{Lemma}
\aliascntresetthe{lemma}
\newaliascnt{proposition}{theorem}
\newtheorem{proposition}[proposition]{Proposition}
\aliascntresetthe{proposition}
\newaliascnt{corollary}{theorem}
\newtheorem{corollary}[corollary]{Corollary}
\aliascntresetthe{corollary}

\theoremstyle{definition}
\newaliascnt{definition}{theorem}
\newtheorem{definition}[definition]{Definition}
\aliascntresetthe{definition}
\newaliascnt{example}{theorem}
\newtheorem{example}[example]{Example}
\aliascntresetthe{example}
\newaliascnt{remark}{theorem}
\newtheorem{remark}[remark]{Remark}
\aliascntresetthe{remark}
\newaliascnt{conjecture}{theorem}
\newtheorem{conjecture}[conjecture]{Conjecture}
\aliascntresetthe{conjecture}
\newaliascnt{problem}{theorem}

\aliascntresetthe{problem}
\newaliascnt{question}{theorem}

\aliascntresetthe{question}

\usepackage{amsfonts,braket,mathtools}
\usepackage{lipsum}
\usepackage{authblk}
\usepackage{blindtext}

\usepackage{tikz-cd}
\usepackage{tikz}
\usetikzlibrary{decorations}
\usetikzlibrary{arrows}
\tikzstyle{v} = [circle, draw, inner sep=2pt, minimum size=3pt, fill=black]
\tikzstyle{l} = [rectangle, draw, rounded corners]

\DeclareMathOperator{\rank}{rank}
\newcommand{\qbinom}[2]{\binom{#1}{#2}_{\hspace{-1.4mm}q}}

\title{$q$-deformation of chromatic polynomials and graphical arrangements}

\author[1]{Tongyu Nian\thanks{u487971i@ecs.osaka-u.ac.jp}}
\author[2]{Shuhei Tsujie\thanks{tsujie.shuhei@a.hokkyodai.ac.jp}}
\author[1]{Ryo Uchiumi\thanks{uchiumi.ryou.1xu@ecs.osaka-u.ac.jp}}
\author[1]{Masahiko Yoshinaga\thanks{yoshinaga@math.sci.osaka-u.ac.jp}}
\affil[1]{Department of Mathematics, The University of Osaka, Toyonaka, Osaka 560-0043, Japan.}
\affil[2]{Department of Mathematics, Hokkaido University of Education, Asahikawa, Hokkaido 070-8621, Japan.}

\providecommand{\keywords}[1]
{
  \small	
  \textbf{\textit{Keywords---}} #1
}

\usepackage[backend=bibtex]{biblatex}
\date{}
\begin{document}
\maketitle
\begin{abstract}
    We first observe a mysterious similarity between the braid arrangement and the arrangement of all hyperplanes in a vector space over the finite field $\mathbb{F}_q$. These two arrangements are defined by the determinants of the Vandermonde and the Moore matrix, respectively. These two matrices are transformed to each other by replacing a natural number $n$ with $q^n$ ($q$-deformation). 

In this paper, we introduce the notion of ``$q$-deformation of graphical arrangements'' as certain subarrangements of the arrangement of all hyperplanes over $\mathbb{F}_q$. This new class of arrangements extends the relationship between the Vandermonde and Moore matrices to graphical arrangements. We show that many invariants of the ``$q$-deformation'' behave as ``$q$-deformation'' of invariants of the graphical arrangements. Such invariants include the characteristic (chromatic) polynomial, the Stirling number of the second kind, freeness, exponents, basis of logarithmic vector fields, etc. 
\end{abstract}


\section{Introduction}
\subsection{Mysterious similarities}

It is classically known that there are similarities between subsets of $[n]=\{1, \dots, n\}$ and 
linear subspaces of $\mathbb{F}_q^n$, which is sometimes called the ``$q$-analogue'' \cite{kacquantum}. 
We start with pointing out further similarities between chromatic polynomials for graphs and 
characteristic polynomials for hyperplane arrangements over finite fields. 

A central arrangement $\mathcal{A}$ is a finite collection of linear hyperplanes in a finite dimensional vector space. 
Define the \textbf{intersection lattice} $L(\mathcal{A})$ and the \textbf{characteristic polynomial} $\chi(\mathcal{A},t)$ by 
\begin{align*}
L(\mathcal{A}) \coloneqq \Set{\bigcap_{H \in \mathcal{B}}H | \mathcal{B} \subseteq \mathcal{A}}, \qquad
\chi(\mathcal{A},t) \coloneqq \sum_{X \in L(\mathcal{A})}\mu(X)t^{\dim X}, 
\end{align*}
where $L(\mathcal{A})$ is ordered by the reverse inclusion and $\mu$ denotes the \textbf{Möbius function} on the lattice $L(\mathcal{A})$. 
See \cite{orlik1992arrangements} for details. 

A typical example of an arrangement is the \textbf{braid arrangement} $\mathcal{B}_{\ell}$ in $\mathbb{R}^{\ell}$ whose defining polynomial is the \textbf{Vandermonde determinant}, i.e., 
\begin{align*}
Q(\mathcal{B}_{\ell}) = \prod_{1 \leq i < j \leq \ell}(x_{j}-x_{i})
= \begin{vmatrix}
1 & x_{1} & x_{1}^{2} & \dots & x_{1}^{\ell-1} \\
1 & x_{2} & x_{2}^{2} & \dots & x_{2}^{\ell-1} \\
\vdots & \vdots & \vdots & & \vdots \\
1 & x_{\ell} & x_{\ell}^{2} & \dots & x_{\ell}^{\ell-1} \\
\end{vmatrix}. 
\end{align*}
The characteristic polynomial of the braid arrangement $\mathcal{B}_{\ell}$ is 
\begin{align*}
\chi(\mathcal{B}_{\ell},t) = t(t-1)(t-2) \cdots (t-\ell+1). 
\end{align*}

There are mysterious similarities between the braid arrangements and the arrangements consisting of all hyperplanes in vector spaces over finite fields.  
Let $q$ be a prime power and $\mathbb{F}_{q}$ the finite field of order $q$. 
Define the arrangement $\mathcal{A}_{\mathrm{all}}(\mathbb{F}_{q}^{\ell})$ as the set of all hyperplanes in $\mathbb{F}_{q}^{\ell}$. 
Its defining polynomial is the determinant of the \textbf{Moore matrix}, i.e., 
\begin{align*}
Q\left(\mathcal{A}_{\mathrm{all}}(\mathbb{F}_{q}^{\ell})\right) 
= \prod_{i = 1}^{\ell}\prod_{c_{1}, \dots, c_{i-1} \in \mathbb{F}_{q}}(c_{1}x_{1} + \dots + c_{i-1}x_{i-1} + x_{i})
=\begin{vmatrix}
x_{1} & x_{1}^{q} & x_{1}^{q^{2}} & \dots & x_{1}^{q^{\ell-1}} \\
x_{2} & x_{2}^{q} & x_{2}^{q^{2}} &  \dots & x_{2}^{q^{\ell-1}} \\
\vdots & \vdots & \vdots &  & \vdots \\
x_{\ell} & x_{\ell}^{q} & x_{\ell}^{q^{2}} & \dots & x_{\ell}^{q^{\ell-1}}
\end{vmatrix}. 
\end{align*}
The characteristic polynomial of $\mathcal{A}_{\mathrm{all}}(\mathbb{F}_{q}^{\ell})$ is 
\begin{align*}
\chi\left(\mathcal{A}_{\mathrm{all}}(\mathbb{F}_{q}^{\ell}),t\right) = (t-1)(t-q)(t-q^{2}) \cdots (t-q^{\ell-1}). 
\end{align*}

By formally replacing $q^{k}$ in the expressions of $Q\left(\mathcal{A}_{\mathrm{all}}(\mathbb{F}_{q}^{\ell})\right)$ and $\chi\left(\mathcal{A}_{\mathrm{all}}(\mathbb{F}_{q}^{\ell}),t\right)$ with $k$, we obtain the expressions for $Q(\mathcal{B}_{\ell})$ and $\chi(\mathcal{B}_{\ell},t)$. 
Note that 
$\chi(\mathcal{B}_{\ell}, \ell)=\ell!=|\mathfrak{S}_\ell|$ and 
$\chi\left(\mathcal{A}_{\mathrm{all}}(\mathbb{F}_{q}^{\ell}),q^\ell\right) = (q^\ell-1)(q^\ell-q) \cdots (q^\ell-q^{\ell-1})=|GL_\ell(\mathbb{F}_q)|$. 
It is worth mentioning that the permutation group $\mathfrak{S}_\ell$ is considered as the ``$\mathbb{F}_1$-version'' of the general linear group $GL_\ell(\mathbb{F}_q)$ \cite{tits1956}. 

A similar phenomenon can be observed in the context of the freeness of arrangements. 
Let $\mathcal{A}=\{H_1, \dots, H_n\}$ be an arrangement of hyperplanes in $V=\mathbb{K}^\ell$. 
Let $\alpha_i: V\to\mathbb{K}$ be a linear form such that $H_i=\alpha_i^{-1}(0)$. The module 
of logarithmic polynomial vector fields $D(\mathcal{A})$ is defined as 
\[
D(\mathcal{A})=
\Set{\theta=\sum_{i=1}^\ell f_i\partial_i| \theta\alpha_i\in(\alpha_i),\ 1\leq i\leq n}, 
\]
where $f_i\in S=\mathbb{K}[x_1, \dots, x_\ell]$ and $\partial_i=\frac{\partial}{\partial x_i}$. 
An arrangement $\mathcal{A}$ is said to be \textbf{free} if 
$D(\mathcal{A})$ is a free module over the polynomial ring (See \cite{orlik1992arrangements} for details). 
Both of $\mathcal{B}_{\ell}$ and $\mathcal{A}_{\mathrm{all}}(\mathbb{F}_{q}^{\ell})$ are free with bases 
\begin{align*}
&\Set{\sum_{i=1}^{\ell}x_{i}^{k}\partial_{i} | 0 \leq k \leq \ell-1} \text{ for } D(\mathcal{B}_{\ell}) \text{ and } \\
&\Set{\sum_{i=1}^{\ell}x_{i}^{q^{k}}\partial_{i} | 0 \leq k \leq \ell-1} \text{ for } D\left(\mathcal{A}_{\mathrm{all}}(\mathbb{F}_{q}^{\ell})\right) 
\end{align*}
(See \cite[Example 4.22 and 4.24]{orlik1992arrangements}). 

Moreover, there are mysterious similarities for subarrangements. 
Let $G$ be a simple graph on $[\ell] = \{1, \dots, \ell\}$. 
Define the \textbf{graphical arrangement} $\mathcal{A}_{G}$ in $\mathbb{R}^{\ell}$ by 
\begin{align*}
\mathcal{A}_{G} \coloneqq \Set{ \{x_{i}-x_{j} = 0\} | \{i,j\} \in E_{G}}. 
\end{align*}
Note that every subarrangement of $\mathcal{B}_{\ell}$ is of the form $\mathcal{A}_{G}$ and it is well known that the chromatic polynomial $\chi(G,t)$ coincides with the characteristic polynomial $\chi(\mathcal{A}_{G},t)$. 
The following proposition for the chromatic polynomial is trivial by definition. 
\begin{proposition}\label{If k colors are not enough then fewer colors are not enough}
Suppose $\chi(G,k) = 0$ for some $k \in \mathbb{Z}_{\geq 0}$. 
Then $\chi(G,j) = 0$ for $0 \leq j \leq k$. 
\end{proposition}

There is a $q$-version of Proposition \ref{If k colors are not enough then fewer colors are not enough}. 

\begin{proposition}[{\cite[Lemma 7]{yoshinaga2007free-potjasams}}]
Let $\mathcal{A}$ be an arrangement in $\mathbb{F}_{q}^{\ell}$. 
If $\chi(\mathcal{A},q^{k}) = 0$ for some $k \in \mathbb{Z}_{\geq 0}$, then $\chi(\mathcal{A},q^{j}) = 0$ for any $0 \leq j \leq k$. 
\end{proposition}
\begin{proof}
Let $\mathcal{A} \otimes \mathbb{F}_{q}^{k}$ denote the subspace arrangement in $(\mathbb{F}_{q}^{k})^{\ell}$ defined by 
\begin{align*}
\mathcal{A} \otimes \mathbb{F}_{q}^{k} \coloneqq \Set{ H \otimes_{\mathbb{F}_{q}} \mathbb{F}_{q}^{k} | H \in \mathcal{A}}. 
\end{align*}
Then the intersection lattices $L(\mathcal{A})$ and $L(\mathcal{A}\otimes\mathbb{F}_{q}^{k}) = \Set{X \otimes_{\mathbb{F}_{q}} \mathbb{F}_{q}^{k} | X \in L(\mathcal{A})}$ are naturally isomorphic. 
By \cite[Proposition 3.1]{bjorner1997subspace-aim}, 
\begin{align*}
\chi(\mathcal{A},q^{k}) = \chi(\mathcal{A}\otimes\mathbb{F}_{q}^{k},q^{k}) 
= \# \left((\mathbb{F}_{q}^{k})^{\ell} \setminus \bigcup_{H \in \mathcal{A}}H\otimes_{\mathbb{F}_{q}}\mathbb{F}_{q}^{k}\right). 
\end{align*}
Thus, if $\chi(\mathcal{A},q^{k})=0$ and $0 \leq j \leq k$, then $\chi(\mathcal{A},q^{j})=0$. 
\end{proof}

We define the \textbf{falling factorial} $t^{\underline{i}}$ for each $i \in \mathbb{Z}_{>0}$ by $t^{\underline{i}} \coloneqq t(t-1) \cdots (t-i+1)$. 
Note that the falling factorial $t^{\underline{\ell}}$ coincides with the characteristic polynomial $\chi(\mathcal{B}_{\ell},t)$. 
\begin{proposition}[{\cite[Theorem 15]{read1968introduction-joct}}]\label{stable partition}
Suppose $\chi(G, t) = \sum_{i=1}^{\ell}c_{i}t^{\underline{i}}$. 
Then $c_{i}$ coincides with the number of stable partitions of $G$ into $i$ blocks, where a stable partition of $G$ is a set partition of the vertex set such that no edge connects vertices within the same block.
In other words, $c_{i}$ coincides with the number of $i$-dimensional subspaces in $L(\mathcal{B}_{\ell})$ that are not contained in any hyperplanes in $\mathcal{A}_{G}$.
\end{proposition}

Next, define the polynomial $t_{q}^{\underline{i}}$ by 
\begin{align*}
t_{q}^{\underline{i}} \coloneqq \chi(\mathcal{A}_{\mathrm{all}}(\mathbb{F}_{q}^{i}),t) = (t-1)(t-q) \cdots (t-q^{i-1}). 
\end{align*}
The following is a $q$-version of \Cref{stable partition}. 
\begin{proposition}\label{q coefficients}
Let $\mathcal{A}$ be an arrangement in $\mathbb{F}_{q}^{\ell}$ and suppose 
$\chi(\mathcal{A},t) = \sum_{i=0}^{\ell}c_{i}t_{q}^{\underline{i}}$. 
Then $c_{i}$ is the number of $i$-dimensional subspaces in $\mathbb{F}_{q}^{\ell}$ that are not contained in any hyperplanes in $\mathcal{A}$. 
\end{proposition}
\begin{proof}
We proceed by double induction on $\ell$ and $|\mathcal{A}_{\mathrm{all}}(\mathbb{F}_{q}^{\ell})\setminus \mathcal{A}|$. 
When $\ell = 1$, 
\begin{align*}
t = (t-1) + 1 &= t^{\underline{1}}_{q} + t^{\underline{0}}_{q} & (\text{for the empty arrangement}),\\
t-1 &= t^{\underline{1}}_{q} & (\text{for the single-point arrangement}). 
\end{align*}
Therefore the claim is true. 

Suppose that $\ell \geq 2$. 
Since $\chi(\mathcal{A}_{\mathrm{all}}(\mathbb{F}_{q}^{\ell}),t) = t^{\underline{\ell}}_{q}$, the assertion is true for $\mathcal{A} = \mathcal{A}_{\mathrm{all}}(\mathbb{F}_{q}^{\ell})$. 
Assume that $\mathcal{A} \subseteq \mathcal{A}_{\mathrm{all}}(\mathbb{F}_{q}^{\ell})$ and $H \in \mathcal{A}$. 
Then, by the deletion-restriction formula \cite[Corollary 2.57]{orlik1992arrangements}, we have  
\begin{align*}
\chi(\mathcal{A}^{\prime}, t) &= \chi(\mathcal{A},t) + \chi(\mathcal{A}^{H},t) 
= \sum_{i=0}^{\ell}c_{i}t^{\underline{i}}_{q} + \sum_{i=0}^{\ell-1}d_{i}t^{\underline{i}}_{q} 
= t^{\underline{\ell}}_{q} + \sum_{i=0}^{\ell-1} (c_{i}+d_{i})t^{\underline{i-1}}_{q}. 
\end{align*}
By the induction hypothesis 
\begin{align*}
c_{i} &= \# \Set{X \in L(\mathcal{A}_{\mathrm{all}}(\mathbb{F}_{q}^{\ell})) | \dim X =i,\ X \not\subseteq K \text{ for any } K \in \mathcal{A} },  \\
d_{i} &= \# \Set{X \in L(\mathcal{A}_{\mathrm{all}}(H)) | \dim X =i,\ X \not\subseteq H \cap K \text{ for any } K \in \mathcal{A}\setminus\{H\} }.  
\end{align*}
Since 
\begin{multline*}
\Set{X \in L(\mathcal{A}_{\mathrm{all}}(\mathbb{F}_{q}^{\ell})) | \dim X = i,\ X \not\subseteq K \text{ for any } K \in \mathcal{A}\setminus\{H\} } \\
= \Set{X \in L(\mathcal{A}_{\mathrm{all}}(\mathbb{F}_{q}^{\ell})) | \dim X =i,\ X \not\subseteq K \text{ for any } K \in \mathcal{A} } \\
\sqcup \Set{X \in L(\mathcal{A}_{\mathrm{all}}(H)) | \dim X =i,\ X \not\subseteq H \cap K \text{ for any } K \in \mathcal{A}\setminus\{H\} }, 
\end{multline*}
the claim follows. 
\end{proof}
Some of combinatorial sequences and their $q$-analogues have 
interpretations in terms of characteristic polynomials of hyperplane 
arrangements over $\mathbb{F}_q$. 
\begin{example}
Consider the empty arrangement in $\mathbb{F}_{q}^{\ell}$. 
Then we have the following well-known formula (See \cite[Expression 4.4]{kacquantum}). 
\begin{align*}
t^{\ell} = \sum_{i=0}^{\ell} \qbinom{\ell}{i} t^{\underline{i}}_{q}, 
\end{align*}
where
\begin{align*}
\qbinom{\ell}{i} = \dfrac{[\ell]_{q} [\ell-1]_{q} \cdots [\ell-i+1]_{q}}{[i]_{q} [i-1]_{q} \cdots [1]_{q}}
\end{align*}
is the $q$-binomial coefficient, which is equal to the number of $i$-dimensional subspaces of $\mathbb{F}_q^\ell$, and $[k]_{q} = \frac{q^{k}-1}{q-1}$ denotes the $q$-integer.  
Note that $t^\ell$ is the characteristic polynomial 
of the empty arrangement in $\mathbb{F}_q^\ell$, 
and $t^{\underline{i}}_{q}$ is the characteristic 
polynomial of $\mathcal{A}_{\mathrm{all}}$ in 
$\mathbb{F}_q^i$. 
\end{example}

\begin{example}
Consider the Boolean arrangement in $\mathbb{F}_{q}^{\ell}$, which consists of all coordinate hyperplanes. 
Define $S_{q}(\ell, i)$ by 
\begin{align*}
(t-1)^{\ell} = \sum_{i=0}^{\ell} (q-1)^{\ell-i}S_{q}(\ell, i)t^{\underline{i}}_{q}, 
\end{align*}
which is called the $q$-Stirling number of the second kind. 
It satisfies the following recurrence formula. 
\begin{align*}
S_{q}(\ell,i) = S_{q}(\ell-1, i-1) + [i]_{q} S_{q}(\ell-1, i), \qquad
S_{q}(0,i) = \delta_{0,i}. 
\end{align*}
It is proved as follows. 
\begin{align*}
(t-1)^{\ell} &= \sum_{i=0}^{\ell-1}(q-1)^{\ell-i-1}S(\ell-1,i)t^{\underline{i}}_{q} (t-1) \\
&= \sum_{i=0}^{\ell-1}(q-1)^{\ell-i-1}S(\ell-1,i)t^{\underline{i}}_{q} (t-q^{i}+q^{i}-1) \\
&= \sum_{i=0}^{\ell-1}(q-1)^{\ell-i-1}S(\ell-1,i)t^{\underline{i+1}}_{q} + \sum_{i=0}^{\ell-1}(q-1)^{\ell-i-1}(q^{i}-1)S_{q}(\ell-1,i)t^{\underline{i}}_{q} \\
&= t^{\underline{\ell}}_{q} + \sum_{i=0}^{\ell -1}(q-1)^{\ell-i} (S_{q}(\ell-1,i-1) + [i]_qS_{q}(\ell-1,i))t^{\underline{i}}_{q}. 
\end{align*}
This is a $q$-version of the following well-known formula for the Stirling number of the second kind. 
\begin{align*}
t^{\ell} = \sum_{i=0}^{\ell}S(\ell,i)t^{\underline{i}}, 
\qquad S(\ell,i) = S(\ell-1, i-1) + i S(\ell-1, i). 
\end{align*}
\begin{remark}
By Proposition \ref{q coefficients}, 
\begin{align*}
S_{q}(\ell,i) = 
\frac{\# \Set{X \in L(\mathcal{A}_{\mathrm{all}}(\mathbb{F}_{q}^{\ell})) | \dim X = i,\ X \not\subseteq \{x_{j}=0\} \text{ for any } j \in [\ell] }}{(q-1)^{\ell-i}}.
\end{align*}
\end{remark}
\end{example}

\subsection{$q$-deformation of graphical arrangements}
We will focus on specific subarrangements of $\mathcal{A}_{\mathrm{all}}(\mathbb{F}_{q}^{\ell})$ that arise from simple graphs.
Note that the braid arrangement $\mathcal{B}_{\ell}$ corresponds to the graphical arrangement associated with the complete graph $K_{\ell}$. 
Thus, it seems natural to assign $K_{\ell}$ to $\mathcal{A}_{\mathrm{all}}(\mathbb{F}_{q}^{\ell})$. 
Furthermore, it seems natural to assign every clique of $G$ to the set of all hyperplanes using the coordinates corresponding to the vertices in the clique. 
\begin{definition}
We define a \textbf{$q$-deformation of graphical arrangement} $\mathcal{A}_{G}^{q}$ in $\mathbb{F}_{q}^{\ell}$ as follows. 
\begin{align*}
\mathcal{A}_{G}^{q} \coloneqq \bigcup_{\{i_{1}, \dots, i_{r}\}} \Set{\{a_{i_{1}}x_{i_{1}} + \dots + a_{i_{r}}x_{i_{r}} = 0\} | (a_{i_{1}}, \dots, a_{i_{r}}) \in \mathbb{F}_{q}^{r}\setminus\{0\}}, 
\end{align*}
where $\{i_{1}, \dots, i_{r}\}$ runs over all cliques of $G$. 
\end{definition}

It is expected that $\mathcal{A}_{G}^{q}$ has a lot of properties similar to the graphical arrangement $\mathcal{A}_{G}$. 
The organization of this paper is as follows. 
In \cref{section:characteristic polynomial}, we will show that the characteristic polynomial of $\mathcal{A}_{G}^{q}$ determines the chromatic polynomial $G$ when $q$ is large enough (\Cref{main theorem} and \Cref{cor}) and we will describe a relationship between the characteristic polynomial and  the numbers of stable partitions of $G$ (\Cref{q stable parition}). 
In \cref{section:freeness}, we will show that $\mathcal{A}_{G}^{q}$ is free if and only if $G$ is chordal as in the case of graphical arrangements (\Cref{freeness for q-graphical arr}). 
In \cref{section:basis}, we construct an explicit basis for $D(\mathcal{A}_{G}^{q})$ for a chordal graph $G$ (\Cref{basis for q-graphical arr}).

\section{Characteristic polynomial}\label{section:characteristic polynomial}

Let $G=(V, E)$ be a graph. One of the most important properties of 
chromatic polynomials is the following deletion-contraction formula 
\[
\chi(G, t)=\chi(G\setminus\{e\}, t)-\chi(G/e, t), 
\]
for $e\in E$. There is a $q$-version of this formula. 
\begin{proposition}
    \label{q-deletion-contraction}
    Let $e=\{i_1, i_2\}\in E$ be an edge such that  
    \begin{enumerate}[label=(\arabic*)]
        \item \label{hyp1} $\{i_1, i_2\}$ is a maximal clique, and 
        \item \label{hyp2} The contraction $G/e$ does not result in a new clique. 
    \end{enumerate}
    Then \[
    \chi(\mathcal{A}_{G}^{q}, t)=\chi(\mathcal{A}_{G\setminus e}^{q}, t)-(q-1)
    \chi(\mathcal{A}_{G/e}^{q}, t). 
    \]
\end{proposition}
\begin{proof}
    Without loss of generality, we may assume that $e=\{\ell-1, \ell\}$. Then the difference 
    between $\mathcal{A}_{G\setminus e}^{q}$ and $\mathcal{A}_{G}^{q}$ is $(q-1)$ hyperplanes 
    of the form $\{x_{\ell-1}=\alpha x_\ell\}$ with $\alpha\in\mathbb{F}_q^\times$. 
    Under the hypothesis \ref{hyp1} and \ref{hyp2}, the restricted arrangement on such a hyperplane 
    is $\mathcal{A}_{G/e}^{q}$. Applying the 
    deletion-restriction theorem \cite[Corollary 2.57]{orlik1992arrangements} $(q-1)$ times, 
    we obtain the formula. 
\end{proof}

\begin{example}
Let $P_\ell$ be the path graph on the vertex set $[\ell]$. 
Then a clique is either a vertex or an edge. Hence 
$\mathcal{A}_{P_\ell}^q$ consists of $1+(\ell-1)q$ hyperplanes 
of the form $\{ax_i+bx_{i+1}=0\}$, $i=1, \dots, \ell-1$, $(a, b)\in\mathbb{F}_q^{2}\setminus\{(0, 0)\}$. 
Using \Cref{q-deletion-contraction}, we have  $\chi(\mathcal{A}_{P_\ell}^q, t)=(t-1)(t-q)^{\ell-1}$ 
(Note that $\chi(\mathcal{A}_{P_\ell}, t)=t(t-1)^{\ell-1}$). 
\end{example}


There is a well-known explicit formula for the chromatic polynomial of a cycle graph $C_\ell$ for $\ell \geq 3$.
\begin{align*}
\chi(\mathcal{A}_{C_{\ell}},t) = (t-1)^{\ell}+(-1)^{\ell}(t-1). 
\end{align*}

The following lemma is a $q$-version of this formula. 
\begin{lemma}\label{characteristic polynomial cycle graph}
Let $\ell \geq 4$. 
Then 
\begin{align*}
\chi(\mathcal{A}_{C_{\ell}}^{q},t) = (t-q)^{\ell}+(-1)^{\ell}(q-1)^{\ell-1}(t-q).
\end{align*}
\end{lemma}
\begin{proof}
First consider the arrangement $\mathcal{B}$ in $\mathbb{F}_{q}^{3}$ consisting of the following hyperplanes. 
\begin{align*}
x_{1}=ax_{2},\ x_{2}=ax_{3},\ x_{3}=ax_{1} \quad (a \in \mathbb{F}_{q}).
\end{align*}
Then $\mathcal{B}$ is supersolvable with exponents $(1,q,2q-1)$. 
Using the deletion-restriction theorem \cite[Corollary 2.57]{orlik1992arrangements} $q-1$ times, we have 
\begin{align*}
\chi(\mathcal{A}_{C_{4}}^{q}, t) &= \chi(\mathcal{A}_{P_{4}}^{q}, t) - (q-1)\chi(\mathcal{B},t) \\
&= (t-1)(t-q)^{3} - (q-1)(t-1)(t-q)(t-2q+1) \\
&= (t-1)(t-q)\left((t-q)^{2}-(q-1)(t-2q+1)\right) \\
&= (t-1)(t-q)\left((t-q)^{2}-(q-1)(t-q) +(q-1)^{2}\right) \\
&= (t-1)(t-q)\dfrac{(t-q)^{3}+(q-1)^{3}}{(t-q)+(q-1)} \\
&= (t-q)^{4}+(-1)^{4}(q-1)^{3}(t-1). 
\end{align*}

Now suppose that $\ell \geq 5$. 
Using \Cref{q-deletion-contraction}, we have 
\begin{align*}
\chi(\mathcal{A}_{C_{\ell}}^{q},t) &= \chi(\mathcal{A}_{P_{\ell}}^{q})-(q-1)\chi(\mathcal{A}_{C_{\ell-1}}^{q},t) \\
&= (t-1)(t-q)^{\ell-1}-(q-1)(t-q)^{\ell-1}-(-1)^{\ell-1}(q-1)^{\ell-1}(t-q) \\
&= (t-q)^{\ell}+(-1)^{\ell}(q-1)^{\ell-1}(t-q). 
\end{align*}
\end{proof}

In many concrete graphs $G$, we observe that 
$\chi(\mathcal{A}_G^q, t)$ is a polynomial in $q$ 
and $t$. 
It is expected that the chromatic polynomial 
$\chi(G, t)$ is recovered from the polynomial 
$\chi(\mathcal{A}_G^q, t)$. 
\begin{conjecture}
\label{conj:polynomiality}
The function
$\chi(\mathcal{A}_G^q, t)$ is a polynomial in 
$q$ and in $t$, and the equality
\[
\lim_{q\to 1}\frac{\chi(\mathcal{A}_G^q, q^t)}{(q-1)^\ell}
=\chi(G, t)
\]
holds. 
\end{conjecture}

The following results are weaker versions 
of this conjecture.

\begin{theorem}
\label{main theorem}
For any $k \in \mathbb{Z}_{\geq 0}$, and prime power $q$, 
\begin{align*}
\dfrac{\chi(\mathcal{A}_{G}^{q},q^{k})}{(q-1)^{\ell}} \equiv \chi(G,k) \pmod{q-1}. 
\end{align*}
\end{theorem}
\begin{proof}

Let $V = \mathbb{F}_{q}^{k}$. 
Define 
\begin{align*}
\underline{\chi}_{G}^{q}(V) \coloneqq \Set{(v_{1}, \dots, v_{\ell}) \in V^{\ell} |\begin{array}{l}
\text{ $\{v_{i_{1}}, \dots, v_{i_{p}}\}$ is linearly independent} \\
\text{ over $\mathbb{F}_{q}$ if $\{i_{1}, \dots, i_{p}\}$ is a clique of $G$.} \\
\end{array}  }.
\end{align*}
Note that $\chi(\mathcal{A}_{G}^{q},q^{k}) = \#\underline{\chi}_{G}^{q}(V)$. 
The group $(\mathbb{F}_{q}^{\times})^{\ell}$ acts on $\underline{\chi}_{G}^{q}(V)$ by 
\begin{align*}
(a_{1}, \dots, a_{\ell}) \cdot (v_{1}, \dots, v_{\ell}) \coloneqq (a_{1}v_{1} ,\dots, a_{\ell}v_{\ell}). 
\end{align*}
Therefore $\# \underline{\chi}_{G}^{q}(\mathbb{P}(V)) = \frac{\chi(\mathcal{A}_{G}^{q}, q^{k})}{(q-1)^{\ell}}$, where 
\begin{align*}
\underline{\chi}_{G}^{q}(\mathbb{P}(V)) \coloneqq \Set{(\overline{v}_{1}, \dots, \overline{v}_{\ell}) \in \mathbb{P}(V)^{\ell} | \begin{array}{l}
\text{ $\{v_{i_{1}}, \dots, v_{i_{p}}\}$ is linearly independent} \\
\text{ over $\mathbb{F}_{q}$ if $\{i_{1}, \dots, i_{p}\}$ is a clique of $G$.} 
\end{array} }. 
\end{align*}

Let $T \coloneqq (\mathbb{F}_{q}^{\times})^{k}/\mathbb{F}_{q}^{\times}(1,1, \dots, 1)$. 
Let $x = \begin{pmatrix}
x_{1} \\ \vdots \\ x_{k}
\end{pmatrix} \neq 0$. 
Then $T$ acts on $\mathbb{P}(V)$ by 
\begin{align*}
\overline{(a_{1}, \dots, a_{k})} \cdot \overline{x} = \overline{\begin{pmatrix}
a_{1}x_{1} \\ \vdots \\ a_{k}x_{k}
\end{pmatrix}}. 
\end{align*}
Let $\nu(x) \coloneqq \#\Set{i \in [k] | x_{i} \neq 0}$. 
Then it is easily seen that 
\begin{align*}
\# (T\cdot \overline{x}) = (q-1)^{\nu(x)-1}.
\end{align*}
In particular, $\overline{x} \in \mathbb{P}(V)$ is a fixed point if and only if $\overline{x} = \overline{e}_{i}$, the standard basis, for some $i$. 
Consider the diagonal $T$-action on $\mathbb{P}(V)^{\ell}$. 
This induces a $T$-action on $\underline{\chi}_{G}^{q}(\mathbb{P}(V))$. 
The fixed point set is 
\begin{align*}
\underline{\chi}_{G}^{q}(\mathbb{P}(V))^{T} = \Set{(\overline{e}_{s_{1}}, \dots, \overline{e}_{s_{\ell}}) | \begin{array}{l}
\text{ $\{e_{s_{i_{1}}}, \dots, e_{s_{i_{p}}}\}$ is linearly independent} \\
\text{ over $\mathbb{F}_{q}$ if $\{i_{1}, \dots, i_{p}\}$ is a clique of $G$.}
\end{array}}.
\end{align*}
Note that $\{e_{s_{i_{1}}}, \dots, e_{s_{i_{p}}}\}$ is linearly independent if and only if they are mutually different. 
Therefore we have 
\begin{align*}
\# \underline{\chi}_{G}^{q}(\mathbb{P}(V))^{T} = \chi(G,k). 
\end{align*}
Any other orbits have cardinalities divisible by $q-1$. 
This completes the proof. 
\end{proof}

\begin{corollary}\label{cor}
If $q > \ell^{\ell}$, then $\chi(\mathcal{A}_{G}^{q},t)$ determines $\chi(G,t)$. 
\end{corollary}
\begin{proof}
Note that $\chi(G,t)$ is determined by the values $\chi(G,1), \chi(G,2), \dots, \chi(G,\ell)$. 
For $k \in [\ell]$, we have $\chi(G,k) \leq k^{\ell} < q-1$. 
Hence $\chi(G,k)$ is the remainder of $\dfrac{\chi(\mathcal{A}_{G}^{q},q^{k})}{(q-1)^{\ell}}$ divided by $q-1$. 
\end{proof}

Recall that a graph $G$ is triangle-free if 
$G$ has no cycle of length $3$. For example, 
bipartite graphs and $C_n$ ($n>3$) are triangle-free. 
\begin{theorem}
Let $G$ be a triangle-free graph. Then, 
\[
\chi(\mathcal{A}_G^q, t)=
(q-1)^\ell\chi(G, \frac{t-1}{q-1}). 
\]
\end{theorem}
\begin{proof}
Let $V=\mathbb{F}_q^k$, and consider 
$\#\underline{\chi}_{G}^{q}(\mathbb{P}(V))$. 
Note that nonzero vectors $v_1, v_2\in V$ are 
linearly independent if and only if 
$\overline{v}_1\neq\overline{v}_2$ in $\mathbb{P}(V)$. 
Since a clique of a triangle-free 
graph $G$ is either a vertex or an edge, 
we have 
\begin{align*}
\underline{\chi}_{G}^{q}(\mathbb{P}(V)) \coloneqq \Set{(\overline{v}_{1}, \dots, \overline{v}_{\ell}) \in \mathbb{P}(V)^{\ell} | 
\text{ $\overline{v}_1\neq\overline{v}_2$ for each 
edge $\{i, j\}\in E$}}. 
\end{align*}
This set is equivalent to the set of vertex coloring 
by the set $\mathbb{P}(V)$. 
Hence, the cardinality is expressed as 
\[
\#\underline{\chi}_{G}^{q}(\mathbb{P}(V))=
\chi(G, \#\mathbb{P}(V))=
\chi(G, \frac{q^k-1}{q-1}). 
\]
Thus we have 
\[
\frac{\chi(\mathcal{A}_G^q, q^k)}{(q-1)^\ell}
=
\chi(G, \frac{q^k-1}{q-1}). 
\]
Once we fix a prime power $q$, both 
$\chi(\mathcal{A}_G^q, t)$ and $(q-1)^\ell \chi(G, \frac{t-1}{q-1})$ are polynomials 
in $t$ of degree $\ell$. 
These polynomials take same value for $t=q^k$ ($k\geq 1$). Hence, they are identical. 
\end{proof}

\begin{remark}
There exists a pair of graphs $(G,H)$ depicted in the following picture such that $\chi(\mathcal{A}_{G}^{2},t) = \chi(\mathcal{A}_{H}^{2},t)$ but $\chi(G,t) \neq \chi(H,t)$. 

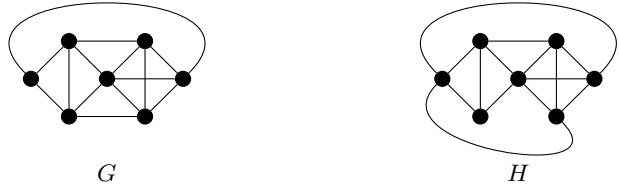
\begin{center}
\begin{tikzpicture}
\draw (0,0) node[v](6){};
\draw (-0.5,0.5) node[v](1){};
\draw (-1,0) node[v](5){};
\draw (-0.5,-0.5) node[v](3){};
\draw (0.5,0.5) node[v](4){};
\draw (0.5,-0.5) node[v](0){};
\draw (1,0) node[v](2){};
\draw (0)--(2);
\draw (0)--(4);
\draw (0)--(6);
\draw (1)--(3);
\draw (1)--(4);
\draw (1)--(5);
\draw (1)--(6);
\draw (2)--(4);
\draw (2)--(6);
\draw (3)--(5);
\draw (3)--(6);
\draw (4)--(6);
\draw (2) to [distance=50] (5);

\draw (0)--(3);
\draw (0,-1.25) node{$G$};
\end{tikzpicture}
\qquad
\begin{tikzpicture}
\draw (0,0) node[v](6){};
\draw (-0.5,0.5) node[v](1){};
\draw (-1,0) node[v](5){};
\draw (-0.5,-0.5) node[v](3){};
\draw (0.5,0.5) node[v](4){};
\draw (0.5,-0.5) node[v](0){};
\draw (1,0) node[v](2){};
\draw (0)--(2);
\draw (0)--(4);
\draw (0)--(6);
\draw (1)--(3);
\draw (1)--(4);
\draw (1)--(5);
\draw (1)--(6);
\draw (2)--(4);
\draw (2)--(6);
\draw (3)--(5);
\draw (3)--(6);
\draw (4)--(6);
\draw (2) to [distance=50] (5);

\draw (5) to [distance=-40] (0);
\draw (0,-1.25) node{$H$};
\end{tikzpicture}
\end{center}
The characteristic polynomials are as follows. 
\begin{align*}
\chi(\mathcal{A}_{G}^{2},t) = \chi(\mathcal{A}_{H}^{2},t) &= t^{7} - 30 t^{6} + 376 t^{5} - 2545 t^{4} + 9934 t^{3} - 21880 t^{2} + 24384 t - 10240, \\
\chi(G,t) &= t^{7} - 14t^{6} + 83t^{5} - 265t^{4} + 474t^{3} - 441t^{2} + 162t, \\
\chi(H,t) &= t^{7} - 14t^{6} + 83t^{5} - 264t^{4} + 468t^{3} - 430t^{2} + 156t. 
\end{align*}
\end{remark}

The characteristic polynomial $\chi(\mathcal{A}_{G}^{q},t)$ can also determine the numbers of stable partitions of $G$ by the following theorem. 

\begin{theorem}\label{q stable parition}
Let $\chi(\mathcal{A}_{G}^{q},t) = \sum_{i=0}^{\ell}c_{i}t_{q}^{\underline{i}}$. 
Then $c_{i}/(q-1)^{\ell-i}$ is a non-negative integer and congruent to the number of stable partitions of $G$ into $i$ blocks modulo $q-1$. 
\end{theorem}
\begin{proof}
By \Cref{q coefficients}, the coefficient $c_{i}$ is equal to the number of subspaces that are not contained in any hyperplanes in $\mathcal{A}_{G}^{q}$. 
Since every $i$-dimensional subspace in $\mathbb{F}_{q}^{\ell}$ corresponds to a matrix consisting of $\ell$ columns with rank $i$ in reduced row echelon form, $c_{i}$ coincides with the number of reduced row echelon form $(v_{1}, \dots, v_{\ell}) \in V^{\ell}$ such that $\{v_{i_{1}}, \dots, v_{i_{p}}\}$ is linearly independent over $\mathbb{F}_{q}$ if $\{i_{1}, \dots, i_{p}\}$ is a clique of $G$, where $V = \mathbb{F}_{q}^{i}$. 
Therefore $c_{i}/(q-1)^{\ell-i}$ equals the number of $(\overline{v}_{1}, \dots, \overline{v}_{\ell}) \in \mathbb{P}(V)^{\ell}$ such that $\{v_{i_{1}}, \dots, v_{i_{p}}\}$ is linearly independent over $\mathbb{F}_{q}$ if $\{i_{1}, \dots, i_{p}\}$ is a clique of $G$. 

Using the same group action in the proof of \Cref{main theorem}, we can show that $c_{i}$ is congruent to the number of reduced row echelon form $(e_{s_{1}}, \dots, e_{s_{\ell}})$ such that $\{e_{s_{i_{1}}}, \dots, e_{s_{i_{p}}}\}$ is linearly independent over $\mathbb{F}_{q}$ if $\{i_{1}, \dots, i_{p}\}$ is a clique of $G$, modulo $q-1$. 
From such a matrix $(e_{s_{i_{1}}}, \dots, e_{s_{i_{\ell}}})$, we can construct a stable partition of $G$ consisting of $i$ blocks $\Set{j \in [\ell] | s_{j}=k} \ (1 \leq k \leq i)$.  
\end{proof}

Some of properties of the chromatic polynomials for finite graphs have $q$-analogues. For example, let us 
denote by $G+K_m$ the join of a graph $G$ and the complete graph $K_m$. Then 
it is easily seen that $\chi(G+K_m, t)=t(t-1)\cdots(t-m+1)\chi(G, t-m)$. 
As a $q$-analogue of this formula, we can prove the following. 
\begin{proposition}\label{conj qC}
Let $G$ be a graph on $[\ell]$. 
Then 
\begin{align*}
\chi(\mathcal{A}_{G+K_m}^{q}, t)=(t-1)(t-q)\cdots (t-q^{m-1})q^{m\ell}\chi(\mathcal{A}_G^{q}, q^{-m}t).
\end{align*}
\end{proposition}

For the proof, we first note that $G+K_m=(G+K_{m-1})+K_1$. It suffices to show the case $m=1$, to be specific,
\begin{align*}
    \chi(\mathcal{A}_{G+K_1}^q,t)=(t-1)q^{\ell}\chi(\mathcal{A}_G^q,q^{-1}t).
\end{align*}
To show this, we shall introduce a new kind of deformation of graphical arrangement, and then complete the proof following a lemma.
\begin{definition}
Let $G$ be a graph on $[\ell]$. 
Define $\widetilde{\mathcal{A}}_{G}^{q}$ of $G$ over the finite field $\mathbb{F}_{q}$ by
\begin{align*}
\widetilde{\mathcal{A}}_{G}^{q} \coloneqq \bigcup_{\{i_{1}, \dots, i_{r}\}} \Set{\{a_{i_{1}}x_{i_{1}} + \dots + a_{i_{r}}x_{i_{r}} = b\} | (a_{i_{1}}, \dots, a_{i_{r}}) \in \mathbb{F}_{q}^{r}\setminus\{0\}, \, b \in \mathbb{F}_{q}}
\end{align*}
where $\{i_{1}, \dots, i_{r}\}$ runs over all cliques of $G$. 
\end{definition}
The following lemma illustrates the characteristic polynomial of this kind of deformation of  graphical arrangement.

\begin{lemma}\label{lemma of qC}
$\chi(\widetilde{\mathcal{A}}_{G}^{q},t)=q^\ell\chi(\mathcal{A}_{G}^{q},q^{-1}t)$.
\end{lemma}

\begin{proof}
By Whitney's theorem \cite[Lemma 2.55]{orlik1992arrangements}, the characteristic polynomial  of $\widetilde{\mathcal{A}}_{G}^{q}$ is represented as 
\begin{align*}
\chi(\widetilde{\mathcal{A}}_{G}^{q},t) = \sum_{\mathcal{B}}(-1)^{|\mathcal{B}|}t^{\dim(\bigcap_{H \in \mathcal{B}}H)}, 
\end{align*}
where $\mathcal{B}$ runs over the subsets of $\widetilde{\mathcal{A}}_{G}^{q}$ which is central, that is, $\bigcap_{H \in \mathcal{B}}H \neq \varnothing$. 
Let $\mathcal{B}$ be a central $k$-subset of $\widetilde{\mathcal{A}}_{G}^{q}$ and suppose that the codimension of the intersection $\bigcap_{H \in \mathcal{B}}H$ is $d$. 
Since every hyperplane in $\widetilde{\mathcal{A}}_{G}^{q}$ is a translate of a hyperplane in $\mathcal{A}_G^{q}$ and $\mathcal{B}$ does not contain parallel hyperplanes, we can associate $\mathcal{B}$ with a $k$-subset of $\mathcal{A}_G^{q}$ by translating hyperplanes in $\mathcal{B}$. 
We will show that this correspondence is $q^{d}$-to-$1$. 


Let $\{H_1', H_2',\dots,H_k'\}$ be a $k$-subset of $\mathcal{A}_{G}^{q}=\{H_1, H_2, \dots, H_n\}$, with 
\begin{align*}
H_i'=\{(x_1,x_2,\dots,x_\ell)\mid a_{i1}x_1+a_{i2}x_2+\cdots+a_{i\ell}x_\ell=0\},
\end{align*}
where $a_{ij}\in\mathbb{F}_q$.
Suppose that the codimension of the intersection $\bigcap_{i=1}^k H_{i}'$ is $d$.
Without loss of generality, we suppose that the first $d$ hyperplanes $H_{1}^{\prime}, \dots, H_{d}^{\prime}$ are independent.
It means that for the $k\times \ell$ matrix $A=(a_{ij})$, all rows can be written as linear combinations of the first $d$ rows and the matrix has rank $d$. Say $r_1,r_2,\dots,r_k$ be rows in $A$, then 
\begin{align*}
r_i=c_{i1}r_1+c_{i2}r_2+\cdots+c_{id}r_d.
\end{align*}
Note that in this case, $H_i'$ corresponds to hyperplanes 
\begin{align*}
\{(x_1,x_2,\dots,x_\ell)\mid a_{i1}x_1+a_{i2}x_2+\cdots+a_{i\ell}x_\ell=b_i\}
\end{align*}
in $\widetilde{\mathcal{A}}_{G}^{q}$ where $b_i\in \mathbb{F}_q$ and $b_i$ are linear combinations of $b_1,b_2,\dots,b_d$ with fixed coefficients, to be specific, 
\begin{align*}
b_i=c_{i1}b_1+c_{i2}b_2+\cdots+c_{id}b_d.
\end{align*}
So we only need to determine the constant terms for the first $d$ hyperplanes in $\widetilde{\mathcal{A}}_{G}^{q}$. Hence we have $q^d$ choices. 

Therefore any subset $I$ of $[n]$ contributes $(-1)^{|I|}q^{\ell-\operatorname{dim}H_I}t^{\operatorname{dim}H_I}$ to the characteristic polynomial of $\widetilde{\mathcal{A}}_{G}^{q}$, 
where $H_I=\bigcap_{i\in I} H_i$ and $H_\varnothing$ is the whole space.
Then we can compute the characteristic polynomial of $\widetilde{\mathcal{A}}_{G}^{q}$ as follows.
    \begin{align*}
        \chi(\widetilde{\mathcal{A}}_{G}^{q},t) &= \sum_{I\subseteq[n]}(-1)^{|I|}q^{\ell-\operatorname{dim}H_I}t^{\dim H_I}\\
        &= \sum_{I\subseteq[n]}(-1)^{|I|}q^{\ell}(q^{-1}t)^{\dim H_I} \\
        &= q^{\ell}\sum_{I\subseteq[n]}(-1)^{|I|}(q^{-1}t)^{\dim H_I}\\
        &= q^\ell\chi(\mathcal{A}_{G}^{q},q^{-1}t).
    \end{align*}
    Hence we get the formula as required.
\end{proof}


We observe that for any clique in $G$, it is also a clique in $G+K_1$ by adding the vertex in $K_1$, so $\widetilde{\mathcal{A}}_{G}^{q}$ is the deconing of $\mathcal{A}_{G+K_1}^q$ with respect to the additional hyperplane $H_0=\{\operatorname{ker}(x_0)\}$ where $x_0$ is the indeterminate corresponding to the added vertex in $K_1$. It is well known that for coning and deconing, the characteristic polynomials differ by a factor $t-1$, which means 
\begin{align*}
    \chi(\mathcal{A}_{G+K_1}^q,t) &= (t-1)\chi(\widetilde{\mathcal{A}}_{G}^{q},t)\\
    &= (t-1)q^\ell\chi(\mathcal{A}_G^q,q^{-1}t)
\end{align*}
as desired. Hence we complete the proof of Proposition \ref{conj qC}.\qed

    


\section{Freeness}\label{section:freeness}

An arrangement $\mathcal{A}$ is \textbf{supersolvable} if and only if there exists a filtration 
\begin{align*}
\varnothing=\mathcal{A}_{0} \subseteq \mathcal{A}_{1} \subseteq \mathcal{A}_{2} \subseteq \dots \subseteq \mathcal{A}_{\ell} = \mathcal{A}
\end{align*}
such that, for each $i \in [\ell]$, $\rank \mathcal{A}_{i} = i$  and for any distinct hyperplanes $H,H^{\prime} \in \mathcal{A}_{i}\setminus\mathcal{A}_{i-1}$ there exists $H^{\prime\prime}$ such that $H \cap H^{\prime} \subseteq H^{\prime\prime}$ \cite[Theorem 4.3]{bjorner1990hyperplane-dcg}. 
Every supersolvable arrangement is inductively free by \cite[Theorem 4.2]{jambu1984free-aim}. 
When $\mathcal{A}$ is supersolvable with the filtration above, the characteristic polynomial decomposes as 
\begin{align*}
\chi(\mathcal{A},t) = \prod_{i=1}^{\ell}(t-|\mathcal{A}_{i}\setminus\mathcal{A}_{i-1}|). 
\end{align*}

A vertex $v$ of a simple graph $G$ is called \textbf{simplicial} if the neighborhood $N_{G}(v) \coloneqq \Set{u \in V_{G} | \{u,v\} \in E_{G}}$ is a clique. 
An ordering $(v_{1}, \dots, v_{\ell})$ of the vertices of $G$ is a \textbf{perfect elimination ordering} if $v_{i}$ is simplicial in the subgraph of $G$ induced by $\{v_{1}, \dots, v_{i}\}$ for each $i \in [\ell]$. 

\begin{theorem}[Stanley (See {\cite[Theorem 3.3]{edelman1994free-mz}}), Dirac \cite{dirac1961rigid-aadmsduh}, Fulkerson-Gross \cite{fulkerson1965incidence-pjom}]\label{graphical arrangement freeness}
The following are equivalent. 
\begin{enumerate}[label=(\arabic*)]
\item\label{PEO} $G$ has a perfect elimination ordering. 
\item $\mathcal{A}_{G}$ is supersolvable. 
\item $\mathcal{A}_{G}$ is free. 
\item\label{chordal} $G$ is chordal. 
\end{enumerate}
Moreover, when $(v_{1}, \dots, v_{\ell})$ is a perfect elimination ordering of $G$
\begin{align*}
\chi(\mathcal{A}_{G}, t) = \prod_{i=1}^{\ell}(t-|N_{G_{i}}(v_{i})|), 
\end{align*}
where $G_{i} \coloneqq G[\{v_{1}, \dots, v_{i}\}]$ and $N_{G_{i}}(v_{i})$ is the set of adjacent vertices of $v_{i}$ in $G_{i}$. 
\end{theorem}
%
%
%

\begin{theorem}\label{freeness for q-graphical arr}
The following conditions are equivalent to the conditions in \Cref{graphical arrangement freeness}. 
\begin{enumerate}[label=(\arabic*)]
\setcounter{enumi}{4}
\item\label{q SS} $\mathcal{A}_{G}^{q}$ is supersolvable. 
\item\label{q free} $\mathcal{A}_{G}^{q}$ is free. 
\end{enumerate}
Moreover, when $(v_{1}, \dots, v_{\ell})$ is a perfect elimination ordering of $G$
\begin{align*}
\chi(\mathcal{A}_{G}^{q}, t) = \prod_{i=1}^{\ell}\left(t-q^{|N_{G_{i}}(v_{i})|}\right), 
\end{align*}
where $G_{i} \coloneqq G[\{v_{1}, \dots, v_{i}\}]$ and $N_{G_{i}}(v_{i})$ is the set of adjacent vertices of $v_{i}$ in $G_{i}$. 
\end{theorem}
\begin{proof}
The proof is very similar to the proof of Theorem \ref{graphical arrangement freeness}. 

To show $\ref{PEO}\Rightarrow\ref{q SS}$, 
let $(v_{1}, \dots, v_{\ell})$ be a perfect elimination ordering and $(x_{1}, \dots, x_{\ell})$ the corresponding coordinates. 
Let $\mathcal{A}_{i}$ be the subarrangement of $\mathcal{A}_{G}^{q}$ consisting hyperplanes whose defining linear form contains $x_{1}, \dots, x_{i}$. 
Then the filtration 
\begin{align*}
\mathcal{A}_{1} \subseteq \mathcal{A}_{2} \subseteq \dots \subseteq \mathcal{A}_{\ell} = \mathcal{A}_{G}^{q}
\end{align*}
guarantees that $\mathcal{A}_{G}^{q}$ is supersolvable and $|\mathcal{A}_{i}\setminus \mathcal{A}_{i-1}| = q^{|N_{G_{i}}(v_{i})|}$. 

The implication $\ref{q SS} \Rightarrow \ref{q free}$ follows since every supersolvable arrangement is inductively free.
To prove $\ref{q free}\Rightarrow\ref{chordal}$ 
It suffices to show that $\chi(\mathcal{A}_{C_{\ell}}^{q},t)$ does not factor into the product of linear forms over $\mathbb{Z}$ when $\ell \geq 4$ by Terao's factorization theorem. 
Since $\chi(\mathcal{A}_{C_{\ell}}^{q},t) = (t-q)^{\ell}+(-1)^{\ell}(q-1)^{\ell-1}(t-q)$ by \Cref{characteristic polynomial cycle graph}, we have 
\begin{align*}
\chi(\mathcal{A}_{C_{\ell}},t) = (q-1)^{\ell}(x^{\ell}+(-1)^{\ell}) = (q-1)^{\ell}x(x^{\ell-1}+(-1)^{\ell}), 
\end{align*}
where $x = \frac{t-q}{q-1}$. 
Since $x^{\ell}+(-1)^{\ell}$ has an imaginary root, $\chi(\mathcal{A}_{C_{\ell}}^{q},t)$ does not factor over $\mathbb{Z}$. 
\end{proof}



\section{Basis construction}\label{section:basis}

Let $G$ be a chordal graph with a perfect elimination ordering $(v_{1}, \dots, v_{\ell})$ and $(x_{1}, \dots, x_{\ell})$ the corresponding coordinates. 
Define the sets $C_{\geq k}$ and $E_{<k}$ by 
\begin{align*}
C_{\geq k} &\coloneqq \{k\} \cup \Set{i \in [\ell] | \text{There exists a path } v_{k}v_{j_{1}}\cdots v_{j_{n}}v_{i} \text{ such that } k < j_{1} < \dots < j_{n} < i}, \\
E_{<k} &\coloneqq \Set{j \in [\ell] | j < k \text{ and } \{v_{j},v_{k}\} \in E_{G}}. 
\end{align*}

Let $\Delta(x_{1}, \dots, x_{k})$ denote the Vandermonde determinant: 
\begin{align*}
\Delta(x_{1}, \dots, x_{k}) \coloneqq \begin{vmatrix}
1 & x_{1} & x_{1}^{2} & \dots & x_{1}^{k-1} \\
1 & x_{2} & x_{2}^{2} &  \dots & x_{k}^{k-1} \\
\vdots & \vdots & \vdots &  & \vdots \\
1 & x_{k} & x_{k}^{2} & \dots & x_{k}^{k-1}
\end{vmatrix} 
= \prod_{1 \leq i < j \leq k}(x_{j}-x_{i}). 
\end{align*}

When $E_{<k} = \{j_{1}, \dots, j_{m}\}$ with $j_{1} < \dots < j_{m}$, 
\begin{align*}
\Delta(E_{<k}) \coloneqq \Delta(x_{j_{1}}, \dots, x_{j_{m}}) 
\quad \text{ and } \quad
\Delta(E_{<k}, x_{i}) \coloneqq \Delta(x_{j_{1}}, \dots, x_{j_{m}}, x_{i}). 
\end{align*}

\begin{theorem}[{\cite[Theorem 4.1]{suyama2019vertex-weighted-dcg}}]
Let 
\begin{align*}
\theta_{k} \coloneqq \sum_{i \in C_{\geq k}}\dfrac{\Delta(E_{<k}, x_{i})}{\Delta(E_{<k})}\partial_{i} \qquad (1 \leq k \leq \ell). 
\end{align*}
Then $\{\theta_{1}, \dots, \theta_{\ell}\}$ forms a basis for $D(\mathcal{A}_{G})$. 
\end{theorem}

Let $\Delta_{q}(x_{1}, \dots, x_{k}) \in \mathbb{F}_{q}[x_{1}, \dots, x_{k}]$ denote the determinant of the Moore matrix. 
Namely 
\begin{align*}
\Delta_{q}(x_{1}, \dots, x_{k}) = \begin{vmatrix}
x_{1} & x_{1}^{q} & x_{1}^{q^{2}} & \dots & x_{1}^{q^{k-1}} \\
x_{2} & x_{2}^{q} & x_{2}^{q^{2}} &  \dots & x_{k}^{q^{k-1}} \\
\vdots & \vdots & \vdots &  & \vdots \\
x_{k} & x_{k}^{q} & x_{k}^{q^{2}} & \dots & x_{k}^{q^{k-1}}
\end{vmatrix} 
= \prod_{i = 1}^{k}\prod_{c_{1}, \dots, c_{i-1} \in \mathbb{F}_{q}}(c_{1}x_{1} + \dots + c_{i-1}x_{i-1} + x_{i}). 
\end{align*}

\begin{theorem}\label{basis for q-graphical arr}
Let 
\begin{align*}
\theta_{k} \coloneqq \sum_{i \in C_{\geq k}}\dfrac{\Delta_{q}(E_{<k}, x_{i})}{\Delta_{q}(E_{<k})}\partial_{i} \qquad (1 \leq k \leq \ell). 
\end{align*}
Then $\{\theta_{1}, \dots, \theta_{\ell}\}$ forms a basis for $D(\mathcal{A}_{G}^{q})$.
\end{theorem}
\begin{proof}
Let $K = \{i_{1}, \dots, i_{m}\}$ be a clique of $G$ with $i_{1} < \dots < i_{m}$. 
Let $\alpha = c_{1}x_{i_{1}} + c_{2}x_{i_{2}} + \dots + c_{m}x_{i_{m}} $ be a nonzero linear form over $\mathbb{F}_{q}$. 
If $K \cap C_{\geq k} = \varnothing$, then $\theta_{k}(\alpha) = 0$. 

Suppose that $K \cap C_{\geq k} \neq \varnothing$ and take $i_{s} \in K \cap C_{\geq k}$ with minimal $s$. 
Then one can show that $\{i_{1}, \dots, i_{s-1}\} \subseteq E_{<k}$ and $\{i_{s}, i_{s+1}, \dots, i_{m}\} \subseteq C_{\geq k}$. 
Then 
\begin{align*}
\theta_{k}(\alpha) 
&= \sum_{u=s}^{k}\dfrac{\Delta_{q}(E_{<k},x_{i_{u}})}{\Delta_{q}(E_{<k})} \cdot c_{i_{u}} 
=\dfrac{\Delta_{q}(E_{<k}, \ c_{i_{s}}x_{i_{s}} + \dots + c_{i_{m}}x_{i_{m}})}{\Delta_{q}(E_{<k})} 
= \dfrac{\Delta_{q}(E_{<k}, \ \alpha)}{\Delta_{q}(E_{<k})} \in (\alpha). 
\end{align*}
Using Saito's criterion \cite[Theorem 4.19]{orlik1992arrangements}, we can prove $\{\theta_{1}, \dots, \theta_{\ell}\}$ is a basis. 
\end{proof}

\section*{Acknowledgement}
S. T. was supported by JSPS KAKENHI, Grant Number JP22K13885. 
R. U. was supported by JST SPRING, Grant Number JPMJSP2138.
M. Y. was partially supported by JSPS KAKENHI, Grant Number JP23H00081.

\printbibliography

\end{document}